\documentclass{article}
\usepackage{amsmath,amsfonts,amssymb,amsthm,graphicx}
 \usepackage{dsfont}    
\usepackage{enumerate}
\usepackage{color}
\usepackage{natbib}
\usepackage{multirow}
\usepackage{comment}

\usepackage{booktabs,subcaption,dcolumn}

\usepackage{algorithm}
\usepackage[noend]{algpseudocode}  
\usepackage[colorlinks=true,citecolor=blue,pdfpagemode=UseNone,pdfstartview=FitH]{hyperref}

\renewcommand{\d}{\,\mathrm{d}}

\newcommand{\p}{\mathbb{P}}

\newcommand{\E}{\mathbb{E}}    
\newcommand{\R}{\mathbb{R}}    

  \newcommand{\id}{\mathds{1}}

\newcommand{\cD}{\mathcal{D}}

\usepackage{setspace}

\theoremstyle{plain}
\newtheorem{theorem}{Theorem} 
\newtheorem{corollary}{Corollary}

\theoremstyle{definition}
\newtheorem{definition}{Definition}

\theoremstyle{remark}
\newtheorem{remark}{Remark}

 \renewcommand{\cite}{\citet}  

\title{Elementary proofs of several results on false discovery rate} 
\author{
  Ruodu Wang\thanks%
  {Department of Statistics and Actuarial Science,
  University of Waterloo,
  Waterloo, Ontario, Canada.
  E-mail: \href{mailto:wang@uwaterloo.ca}{wang@uwaterloo.ca}.}}

\begin{document}

\maketitle

\begin{abstract}
We collect self-contained elementary proofs of four results in the literature on the false discovery rate  of the Benjamini-Hochberg (BH) procedure for independent or positive-regression dependent p-values, the Benjamini-Yekutieli correction for arbitrarily dependent p-values, and the e-BH procedure for arbitrarily dependent e-values. As a corollary, the above proofs also lead to some inequalities of Simes and Hommel.
\end{abstract}
 
 \section{False discovery rate} 

This short note contains proofs of four   results  on controlling the false discovery rate (FDR) using the Benjamini-Hochberg (BH), Benjamini-Yekutieli (BY) and e-BH procedures proposed by, respectively, \cite{BH95}, \cite{BY01} and \cite{WR22}. 
All results are known and   various proofs exist in the literature. Our proofs are elementary, concise,  and self-contained. They may not be useful for experts in the field, but they may become helpful for someone who is trained in probability theory but not much in multiple hypothesis testing (e.g., a graduate student). This note serves pedagogical purposes and does not contain any new results.

We first briefly define false discovery rate in multiple hypothesis testing.
Let $H_1,\dots,H_K$ be $K$ hypotheses, and  denote by $\mathcal K=\{1,\dots,K\}$.
The true data-generating  probability measure is denoted by $\p$.
For each $k\in \mathcal K$,  $H_k$ is called a true null hypothesis if $\p \in H_k$, and 
the set $\mathcal N\subseteq \mathcal K$ be the set of indices of true null hypotheses, which is unknown to the decision maker. Let $K_0=|\mathcal N|$ be the number of true null hypotheses.
A \emph{p-variable} $P$ is a  random variable  that satisfies $\p(P\le \alpha)\le \alpha$ for all $\alpha \in (0,1)$.  
 For each $k\in \mathcal K$,  
  $H_k$ is  associated with p-value $p_k$, which is a realization of a random variable $P_k$, where $P_k$ is a p-variable if $k\in \mathcal N$.
For convenience, we will use the term  ``p-value"  for  both the random variable  $P_k$ and its realized value. We always write $\mathbf P=(P_1,\dots,P_K)$, with the capital letters emphasizing that they are random variables. We also encounter e-values (\cite{VW21}) which will be defined in Section \ref{sec:4}.

 Let $\cD:[0,\infty)^K\to 2^{\mathcal K}$ be a \emph{testing procedure}, which reports the indices of rejected hypotheses based on observed p-values.  The general aim is to reject non-null hypotheses among $H_1,\dots,H_K$.
Write  $R_{\cD}:=|\cD|$, the number of discoveries (rejected hypotheses), and $F_{\cD}:=|\cD \cap \mathcal N|$, the number of false discoveries and $R_{\cD}$ is  
 The ratio   $F_{\cD}/R_{\cD}$ is called the false discovery proportion (FDP) with the convention $0/0=0$. 
  The quantities $\cD$, $F_{\cD}$ and $R_{\cD}$ are treated as random objects as they are functions of the p-values. 
   \cite{BH95} proposed to control FDR defined by
  $ 
 \mathrm{FDR}_{\cD}:=\E [ {F_{\cD}}/({R_{\cD} \vee 1}) ],
$ 
where the expected value is taken  under  $\p$.

The assumption on the p-values  or  the e-values  varies by sections. In Section \ref{sec:2}, null p-values are iid uniform and independent of the non-null p-values.
In Section \ref{sec:3}, p-values satisfy a notion of positive dependence. In Sections \ref{sec:4} and \ref{sec:5}, we deal with arbitrarily dependent e-values and p-values, respectively. 
   
 \section{The BH procedure for iid null p-values}\label{sec:2}

For $k\in \mathcal K$, let $p_{(k)}$ be the $k$-th order statistics of $p_1,\dots,p_K$, from the smallest to the largest. 
The \emph{BH procedure}  at level $\alpha \in (0,1)$ rejects all hypotheses with the smallest $k^*$ p-values,
where 
$$
k^*=\max\left\{k\in \mathcal K: \frac{K p_{(k)}}{k} \le \alpha\right\},
 $$
with the convention $\max(\varnothing)=0$, 
and accepts the rest. 
 
 We provide three proofs of  the main result of \cite{BH95}. 
 The first one uses an argument  based on the optional stopping theorem. This proof was provided by \citet[Chapter 5]{S02}.
The second one is based on a replacement technique. The third one uses the argument from the proof of Theorem \ref{th:BH-PRDS} below as explained in Remark \ref{rem:PRDS}.
  \begin{theorem}\label{th:BH}
  If the null p-values are iid uniform on $[0,1]$ and independent of the non-null p-values, then 
 BH procedure at level $\alpha\in (0,1) $ has FDR  equal to $\alpha K_0/K$.
 \end{theorem}   
 \begin{proof}[First proof of Theorem \ref{th:BH}]
 Let $\cD$ be the BH procedure at level $\alpha$.  
 For $t\in [0,1]$, 
 let $ F(t)=|\{k\in \mathcal N:P_k \le t\} | $, $R(t) =|\{k\in \mathcal K:P_k \le t\}|\vee 1,$ and 
 $$
 t_\alpha = \sup \{t\in [0,1]:{Kt} \le \alpha R(t)\}.
 $$
Since $R$ is upper semicontinuous, we know $K t_\alpha =\alpha R(t_\alpha)$. 
Take $k\in \mathcal K$. If   $P_k\le t_\alpha$,
then there exists   $\ell \ge  R(P_k)$ 
such that $
{Kp_{(\ell)}}/{\ell} \le \alpha, 
$
which means  $\ell\le k^*$, and hence $H_k$ is rejected by the BH procedure. 
If $P_k > t_\alpha$, then 
 ${KP_k} > \alpha{R(P_k)} ,$ 
 and hence $H_k$ is not rejected by the BH procedure. 
To summarize,    each $H_k$ is rejected by the BH procedure if and only if $P_k\le t_\alpha$.  

For $t\in [0,1]$, let
 $$ 
\mathcal F_t =\sigma \{(\id_{\{P_1\le s\}},\dots, \id_{\{P_K\le s\}}): s
\in [ t,1]\}.
 $$
Note that for  $k\in \mathcal N$ and $s\le t$, $$
\E[\id_{\{P_k\le s\}} \mid \mathcal F_t] =\E[\id_{\{P_k\le s\}} \mid \id_{\{P_k\le t\}}] = \frac{s}{t}\id_{\{P_k\le t\}}.
$$ 
Hence, for $s\le t$, we have $\E[F(s)|\mathcal F_t] =sF(t)/t$, 
which implies that $ t\mapsto F(t)/t$ is a backward martingale.
Note that  $t_\alpha$ is a stopping time with respect to the filtration  $(\mathcal F_t)_{t\in [0,1]}$. 
 The optional stopping theorem gives
 $\E[F(t_\alpha)/t_\alpha] = \E[F(1)] =F(1) = K_0$. 
 Hence, using $F_\cD = F(t_\alpha)$ and  $R_\cD =  R(t_\alpha)$,
   $$
\E\left[\frac{F_{\cD}}{R_{\cD} }\right] =
\E\left[\frac{ F(t_\alpha) }{R(t_\alpha)}\right]   
=\frac{\alpha}{K}
\E\left[\frac{F(t_\alpha)}{t_\alpha}\right] 
= \frac{ K_0  }{K  }\alpha,
$$
and this completes the proof.  
 \end{proof}

 \begin{proof}[Second proof of Theorem \ref{th:BH}] 
 Let $\alpha_r= \alpha r/K$ for $r\in \mathcal K$.
We can write
\begin{align}\label{eq:rep1}
\E\left[\frac{F_{\cD  }}{R_{\cD }   }\right]
& 
=\E\left[ \frac{ \sum_{k\in \mathcal N} \id_{\{P_k\le \alpha_{R_\cD}\}}}{R_\cD}\right]
  =\sum_{k\in \mathcal N} \sum_{r=1}^K \frac {1}{r} \E\left[   \id_{\{P_k\le \alpha_{r}\}}\id_{\{R_\cD=r\}}\right]  .
\end{align} 
For $k\in \mathcal N$, 
let $R_k$ be the number of rejection from the BH procedure if it is applied to $\mathbf P$ with $P_k$ replaced by $0$. 
Note that $\{ P_k\le \alpha_{r},~R_\cD=r\} =\{ P_k\le \alpha_{r},~ R_k=r\}$ for each $k,r$.
 Hence, we have 
 $$
 \E[\id_{\{P_k\le \alpha_{r}\}}   \id_{\{R_\cD = r\}}] 
   = 
  \E[\id_{\{P_k\le \alpha_{ r }\}}   \id_{\{  R_k  = r\}}].
 $$
 For $k\in \mathcal N$, 
independence between $P_k$  and  $(P_j)_{j\in \mathcal K\setminus \{k\}}$ implies that 
 $P_k$ and $R_k$ are independent. Hence,
\begin{align}
 \E[\id_{\{P_k\le \alpha_{r}\}}   \id_{\{R_k = r\}}] 
 & =\p(P_k\le \alpha_{r}) \p(R_j=r) = \frac{\alpha  r }{K} \p(R_k=r).\label{eq:independent}
\end{align}
Putting this into \eqref{eq:rep1}, we get
 \begin{align}
 \label{eq:rep3}
\E\left[\frac{F_{\cD  }}{R_{\cD }   }\right] = \frac{\alpha}{K} \sum_{k\in \mathcal N} \sum_{r=1}^K  \p(R_k=r)= \frac{K_0 \alpha}{K},
\end{align} and this completes the proof. 
 \end{proof}

    \section{The BH procedure for PRDS p-values}
    \label{sec:3}

In what follows, inequalities should be
interpreted component-wise when applied to vectors, and terms like ``increasing" or ``decreasing" are in the non-strict sense. \begin{definition}\label{def:PRD}
  A set $A\subseteq \R^K$ is said to be \emph{increasing} 
if $\mathbf x\in A$ implies $\mathbf y\in A$ for  all $\mathbf y\ge \mathbf x$. 
 The p-values $P_1,\dots,P_K$ satisfy \emph{positive regression dependence on the subset $\mathcal N$ (PRDS)} if for any null index $k\in \mathcal N$ and   increasing set $A  \subseteq  \R^K$, the
function $x\mapsto \p(\mathbf P\in A\mid P_k\le x)$ is increasing.
\end{definition}  
This version of PRDS is used by \cite[Section 4]{FDR09} which is weaker than the original one used in \cite{BY01}, where ``$P_k \leq x$" in Definition \ref{def:PRD} is replaced by ``$P_k=x$"; see also Lemma 1 of \cite{RBWJ19}.     
Below, we present a proof of  the FDR guarantee of the BH procedure for PRDS p-values.  This proof is found in \cite{FDR09}. 
  \begin{theorem}\label{th:BH-PRDS} 
 If   the p-values satisfy PRDS, then 
the  BH procedure at level $\alpha $ has FDR at most $\alpha K_0/K$.
 \end{theorem}   
 \begin{proof}
 Let $\cD$ be the BH procedure at level $\alpha$.  
 Write $\alpha_r=\alpha r  /K$   and $\beta_{k,r} = \p( R_{\cD} \ge  r \mid  P_k \le \alpha_r) $ for $r, k\in \mathcal K$, and set $\beta_{k,K+1}=0$.
By noting that $  R_{\cD} $ is a  decreasing function of the p-values,   the PRDS property gives, for $k\in \mathcal N$ and $r\in\mathcal K$,
\begin{align}\label{eq:ineq1}
  \p( R_{\cD} \ge  r +1\mid  P_k \le \alpha_r)  \ge \p( R_{\cD} \ge  r +1\mid  P_k \le \alpha_{r+1})  =\beta_{k,r+1},
 \end{align}
 which leads to
 \begin{align*}   \beta_{k,r} - \beta_{k,r+1}   
&  \ge  \p( R_{\cD} \ge  r \mid  P_k \le \alpha_r) - \p( R_{\cD} \ge  r +1\mid  P_k \le \alpha_r) 
 \\ & =  \p( R_{\cD} = r \mid  P_k \le \alpha_r)  .
\end{align*} 
Using this inequality and \eqref{eq:rep1}, we get 
 \begin{align} 
\E\left[\frac{F_{\cD}}{R_{\cD} }\right]  &=
\sum_{k\in \mathcal N}  \sum_{r=1}^K \frac{1}{r} \p(P_k \le \alpha_r) \p( R_{\cD } = r \mid  P_k \le \alpha_r)  \notag 
\\& \le 
\sum_{k\in \mathcal N}  \sum_{r=1}^K \frac{\alpha }{ K }\p( R_{\cD } = r \mid  P_k \le \alpha_r) \label{eq:ineq2} 
\\& \le 
\sum_{k\in \mathcal N}  \sum_{r=1}^K \frac{\alpha }{ K } (\beta_{k,r} - \beta_{k,r+1}   ) \notag  =
\sum_{k\in \mathcal N}   \frac{\alpha }{ K } \beta_{k,1}   = \frac{\alpha K_0 }{ K}, 
\end{align} 
where the inequality \eqref{eq:ineq2} follows from $\p(P_k \le \alpha_r)\le \alpha_r$, and the last equality follows from $\beta_{k,1}=1$.
 \end{proof}
 \begin{remark}\label{rem:PRDS}
 In the proof of Theorem \ref{th:BH-PRDS}, the inequality \eqref{eq:ineq1} 
 holds  as an equality if $P_r$ is independent of $(P_k)_{k\in \mathcal K\setminus \{r\}}$,
 and  the inequality \eqref{eq:ineq2}   holds  as an equality  if $P_k$ is uniformly distributed. Therefore, this argument also proves Theorem \ref{th:BH}. 
See also   \cite{RBWJ19}  for a superuniformity lemma which becomes useful when showing other FDR statements for PRDS p-values.
 \end{remark}

Next, we mention the Simes inequality as a corollary of Theorem \ref{th:BH}-\ref{th:BH-PRDS}.
 Define the function $S_K:[0,\infty)^K\to [0,\infty)$ of \cite{S86}  as 
$$
  S_K(p_1,\dots,p_k)=\bigwedge_{k=1}^K\frac {K }{k}p_{(k)},
$$
where $p_{(k)}$ is the $k$-th smallest order statistic of $p_1,\dots,p_K$.  
The Simes function is closely linked to the BH procedure.  In case $\mathcal N=\mathcal K$, that is, the global null, any rejection is a false discovery, and the FDP is $1$ as soon as 
there  is any rejection. Therefore,  the FDR of the BH procedure $\cD$ at level $\alpha$ is equal to   $\p(\cD  \ne \varnothing)$, and  $\cD \ne \varnothing $ is equivalent to 
$S_K(\mathbf P) \le \alpha $. 
The Simes inequality, shown by \cite{S86} for the iid case and \cite{S98} for a notion  of positive dependence, follows directly from Theorems \ref{th:BH}-\ref{th:BH-PRDS} and the above observation in the case of PRDS p-values.
\begin{corollary}\label{coro:simes}
Suppose $\mathcal N=\mathcal K$. If the p-values are PRDS, then
$$\p( S_K(\mathbf P) \le \alpha ) \le \alpha \mbox{~~~~for all $\alpha \in [0,1]$.}$$
If the p-values are iid uniform on $[0,1]$, then $S_K(\mathbf P)$ is also uniform on $[0,1]$.
\end{corollary} 
The last statement on uniformity  in Corollary \ref{coro:simes} was originally shown by \cite{S86} using a concise proof by induction.

\section{The e-BH procedure for arbitrary e-values}
\label{sec:4}
As introduced by \cite{VW21},  
an \emph{e-variable} $E$ is a $[0,\infty]$-valued random variable satisfying $\E[E]\le1$. 
E-variables are often obtained from stopping an \emph{e-process} $(E_t)_{t \geq 0}$, which is a nonnegative stochastic process adapted to a pre-specified filtration such that $\mathbb{E}[E_\tau] \leq 1$ for any stopping time $\tau$. 
In the setting of e-values, for each $k\in \mathcal K$, 
  $H_k$ is  associated with e-value $e_k$, which is a realization of a random variable $E_k$ being an e-variable if $k\in \mathcal N$, and an e-testing procedure $\cD: [0,\infty]^K\to 2^{\mathcal K}$ takes these e-values as input.

For $k\in \mathcal K$, let $e_{[k]}$ be the $k$-th order statistic of $e_1,\dots,e_K$, sorted from {the largest to the smallest} so that $e_{[1]}$ is the largest e-value. 
The e-BH procedure at level $\alpha \in (0,1)$, proposed by \cite{WR22}, rejects hypotheses with the largest $k^*$ e-values,
where 
 $$
k^*=\max\left\{k\in \mathcal K: \frac{k e_{[k]}}{K} \ge \frac{1}{\alpha}\right\}.
 $$
 In other words, the e-BH procedure is equivalent to the BH procedure applied to $(e_1^{-1},\dots,e_K^{-1})$. 
  
Below we present a simple proof of the FDR guarantee of the e-BH procedure. \cite{WR22} also considered the e-BH procedure on boosted e-values which we omit here.
\begin{theorem}\label{th:e-BH}
For arbitrary e-values, the e-BH procedure at level $\alpha\in (0,1) $ has FDR at most $\alpha K_0/K$.
\end{theorem}
 
\begin{proof} 
Let $\cD$ be the e-BH procedure at level $\alpha$. By definition, the   e-BH procedure $\cD$ applied to arbitrary e-values $(E_1,\dots,E_K)$ satisfies the following property
\begin{equation}
E_k \ge \frac{K}{\alpha R_{\cD}  } \mbox{~~~if  $k\in \cD$}. \label{eq:sc} 
\end{equation}
Using \eqref{eq:sc}, the FDP of $\cD$ satisfies 
\begin{align*} 
 \frac{F_{\cD}}{R_{\cD} }   =
 \frac{ |\cD \cap \mathcal N |  }{R_{\cD} \vee 1}  
 & =\sum_{k\in \mathcal N}  \frac{ \id_{\{k\in \cD \}}  }{R_{\cD} \vee 1}  
  \le \sum_{k\in \mathcal N}  \frac{ \id_{\{k\in \cD \}} \alpha E_k   }{K}   \le \sum_{k\in \mathcal N}   \frac{\alpha   E_k  }{K} , 
\end{align*} 
where the first inequality  is due to \eqref{eq:sc}. 
As $\E[E_k]\le 1$ for $k\in \mathcal N$, we have
$$
\E\left[ \frac{F_{\cD}}{R_{\cD} }  \right]  \le   \sum_{k\in \mathcal N}  \E\left[ \frac{\alpha   E_k  }{K}   \right] \le \frac{\alpha K_0}{K},
$$
thus the desired FDR guarantee. 
\end{proof}

\begin{remark}
Any e-testing procedure $\cD$   satisfying \eqref{eq:sc}  is said to be self-consistent by \cite{WR22}. The   e-BH procedure dominates all other self-consistent e-testing procedures.
From the proof of Theorem \ref{th:e-BH}, any self-consistent e-testing procedure  has FDR at most $\alpha K_0/K$ for arbitrary   e-values. 
\end{remark}

\section{The  BY correction for arbitrary p-values}
\label{sec:5}
 
 The next theorem concerns the FDR guarantee of the BH procedure for arbitrary p-values.
 As shown by \cite{BY01},  in the most adversarial scenario,
 the BH procedure needs to pay the price of a factor of $\ell_K$, where 
\[
\ell_K=\sum_{k=1}^K \frac 1k \approx \log K.
 \] 
 We provide two simple proofs. The first one is similar to the original proof of \cite{BY01}. 
 The second proof is based on the FDR of  the e-BH procedure in Theorem \ref{th:e-BH}; a similar argument, without explicitly using e-values, is given by \cite{BR08}.
   \begin{theorem}\label{th:BY} 
  For arbitrary p-values,
the  BH procedure at level $\alpha \in (0,1) $ has FDR at most $\ell_K\alpha K_0/K$.
 \end{theorem}

 \begin{proof}[First proof of Theorem \ref{th:BY}]
 Note that \eqref{eq:rep1} can be rearranged to
\begin{align}\label{eq:rep2}
\E\left[\frac{F_{\cD  }}{R_{\cD }   }\right]& = 
\sum_{k\in \mathcal N}   \left( \sum_{r=1}^K     \frac{  \E[ \id_{\{P_k\le \alpha_{R_\cD}\}}  \id_{\{R_\cD \le r\}}] }{r (r+1) }   + \frac{\E [  \id_{\{P_k\le \alpha_{R_\cD}\}}  ]}{K+1}
\right) . 
\end{align} 
For $k\in \mathcal N$, 
 $$
 \E\left [\id_{\{P_k\le \alpha_{R_\cD}\}}   \id_{\{R_\cD \le r\}}\right]
\le  \E[\id_{\{P_k\le \alpha_r\}}   \id_{\{R_\cD \le r\}} ] \le  \E[\id_{\{P_k\le \alpha_r\}}     ] \le \frac{r  }{K}\alpha
, $$
and hence \eqref{eq:rep2} leads to
\begin{align*}
\E\left[\frac{F_{\cD  }}{R_{\cD }   }\right]&\le  
  \sum_{r=1}^K     \frac{ K_0 \alpha }{ (r+1) K }   + \frac{ K_0 \alpha}{K+1} 
  \\
&  =  \sum_{r=1}^K     \frac{ K_0 \alpha }{ (r+1) K }   + \frac{ K_0 \alpha}{K}   -\frac{ K_0 \alpha}{K(K+1)}   = \sum_{r=1}^K  \frac{ 1 }{r   } \frac{ K_0 \alpha }{  K }  =\frac{ \ell _K K_0  }{  K }\alpha  ,
\end{align*}  
and this completes the proof. 
 \end{proof}

 \begin{proof}[Second proof of Theorem \ref{th:BY}]
 Denote by $\alpha'=\alpha \ell_K$.
 Define the decreasing function $\phi: [0,\infty) \to [0,\alpha' K]$ by
\begin{align}\label{eq:transform}
 \phi(p) =  \frac{  K }{ \alpha'  \lceil  K p /\alpha \rceil}\id_{\{ p\le \alpha\}}\mbox{~~with }\phi(0)={ K}/{\alpha' }.
 \end{align}
 Let $E_k=\phi(P_k)$ for $k\in \mathcal K$. 
It is straightforward to check that, for  
$\int_0^1 \phi(p)\d p=1$. 
Therefore, $E_k $ is an e-value for $k\in \mathcal N$. 
By Theorem \ref{th:e-BH}, applying the e-BH procedure at level $\alpha'$ to $(E_1,\dots,E_K)$ has an FDR guarantee of $\alpha' K_0/K$. 
Note that 
$$\frac{k E_{[k]}}{K} \ge \frac{1}{\alpha'}
~~\Longleftrightarrow~~\frac{  k }{ \lceil  K P_{(k)} /\alpha \rceil}  \ge 1 ~~\Longleftrightarrow~~ \frac{K P_{(k)}}{k} \le \alpha.
$$
Hence, the e-BH procedure  at level $\alpha'$  applied to $(E_1,\dots,E_K)$  is equivalent to the BH procedure applied to $(P_1,\dots,P_K)$ at level $\alpha$. 
This yields the FDR guarantee  $ \alpha' K_0/K$  of the BH procedure.   
 \end{proof}

The second proof of Theorem \ref{th:BY} can be generalized to many other procedures, as long as they can be converted to an e-BH procedure via a transform similar to \eqref{eq:transform}.

In a similar way to Corollary \ref{coro:simes}, the inequality of \cite{H83} for arbitrary p-values follows from Theorem \ref{th:BY}.
\begin{corollary}\label{coro:hommel}
Suppose $\mathcal N=\mathcal K$. For arbitrary p-values, we have
$$\p( S_K(\mathbf P) \le \alpha ) \le \ell_K \alpha \mbox{~~~~for all $\alpha \in [0,1]$.}$$ 
\end{corollary} 
As observed by \cite{H83} and \cite{S86}, the inequalities in Corollaries \ref{coro:simes} and \ref{coro:hommel} cannot be improved to obtain a smaller upper bound under their respective assumptions. Similarly, the FDR upper bound in Theorems \ref{th:BH}-\ref{th:BY} cannot be improved in general.

\subsection*{Acknowledgements} 
I thank  Yuzo Maruyama for a useful comment on the presentation of the second proof of Theorem \ref{th:BH}.


\begin{thebibliography}{10}

 \small
 \bibitem[\protect\citeauthoryear{Benjamini and Hochberg}{1995}]{BH95}
  Benjamini, Y. and  Hochberg, Y. (1995). Controlling the false discovery rate: A practical and powerful approach to multiple testing. \emph{Journal of the Royal Statistical Society Series B}, \textbf{57}(1), 289--300.



 \bibitem[\protect\citeauthoryear{Benjamini and Yekutieli}{Benjamini and
   Yekutieli}{2001}]{BY01}
 Benjamini, Y. and Yekutieli, D. (2001).
   The control of the false discovery rate in multiple testing under
   dependency.
   {\em Annals of Statistics}, \textbf{29}(4), 1165--1188.
   
    \bibitem[\protect\citeauthoryear{Blanchard and Roquain}{Blanchard and Roquain}{2008}]{BR08} 
Blanchard, G. and Roquain, E. (2008).
  Two simple sufficient conditions for {FDR} control.
 {\em Electronic Journal of Statistics}, \textbf{2}, 963--992.

%
%



\bibitem[\protect\citeauthoryear{Finner et~al.}{Finner et~al.}{2009}]{FDR09}
Finner, H., Dickhaus, T. and Roters, M. (2009).
 On the false discovery rate and an asymptotically optimal rejection
  curve. {\em The Annals of Statistics}, \textbf{37}(2), 596--618.

  \bibitem[\protect\citeauthoryear{Hommel}{Hommel}{1983}]{H83}
Hommel, G. (1983).
 Tests of the overall hypothesis for arbitrary dependence structures.
 {\em Biometrical Journal}, {\bf 25}(5), 423--430.

%


%




\bibitem[\protect\citeauthoryear{Ramdas et al.}{Ramdas et al.}{2019}]{RBWJ19}
Ramdas, A.~K., Barber, R.~F., Wainwright, M.~J. and Jordan, M.~I. (2019).   A unified treatment of multiple testing with prior knowledge using the p-filter.    \emph{Annals of Statistics}, \textbf{47}(5), 2790--2821.

  \bibitem[\protect\citeauthoryear{Sarkar}{Sarkar}{1998}]{S98}
Sarkar, S.~K. (1998).
  Some probability inequalities for ordered MTP2 random variables: A
  proof of the Simes conjecture.
  {\em Annals of Statistics}, \textbf{26}(2), 494--504.
  

\bibitem[\protect\citeauthoryear{Simes}{Simes}{1986}]{S86}
Simes, R. J. (1986).
  An improved {B}onferroni procedure for multiple tests of
  significance.
  {\em Biometrika}, \textbf{73}, 751--754.
  
 \bibitem[\protect\citeauthoryear{Storey}{2002}]{S02}
 Storey, J. (2002). \emph{False Discovery Rates: Theory and Applications to DNA Microarrays}. PhD Thesis, Stanford University. 
  
  


%

\bibitem[\protect\citeauthoryear{Vovk and Wang}{Vovk and Wang}{2021}]{VW21}
Vovk, V. and  Wang, R. (2021).
E-values: Calibration, combination, and applications.  \emph{Annals of Statistics}, \textbf{49}(3), 1736--1754.

 \bibitem[\protect\citeauthoryear{Wang and Ramdas}{Wang and Ramdas}{2022}]{WR22}
Wang, R. and Ramdas, A. (2022).
False discovery rate control with e-values.  \emph{Journal of the Royal Statistical Society Series B}, \textbf{84}(3), 822--852. 
 \end{thebibliography}
\end{document}